\documentclass[smallextended]{svjour3}

\usepackage{graphicx}
\usepackage{psfrag} 
\usepackage{mathptmx}      
\usepackage{amssymb}
\usepackage{amsmath}
\usepackage{xcolor}
\usepackage[numbers]{natbib}
\usepackage{marvosym}
\usepackage{a4wide}

\usepackage{color}

\usepackage[mathscr]{euscript}

\newcommand{\cO}{{\frak O}}

\journalname{Bulletin of Mathematical Biology}

\begin{document}
		\title{Forest-based networks}
		
\author{K.\,T.\,Huber \and V.\,Moulton \and G.\, E. \,Scholz}
\institute{K.T. Huber \at
	UEA, Norwich, UK. \\
	\email{k.huber@uea.ac.uk}
	\and 
V. Moulton \at
UEA, Norwich, UK. \\
\email{k.huber@uea.ac.uk}
\and 
G. Scholz \at
Bioinformatics Group, Department of Computer Science, Interdisciplinary
Center for Bioinformatics, Leipzig University, Leipzig, Germany.\\
\email{gllm.scholz@gmail.com}
}

\date{\today}
\maketitle

\begin{abstract}
In evolutionary studies it is common to use phylogenetic trees to represent the evolutionary history of a set of species. However, in case the transfer of genes or other genetic information between the species or their ancestors has occurred in the past, a tree may not provide a complete picture of their history. In such cases {\em tree-based phylogenetic networks} can provide a useful, more refined representation of the species’ evolution. Such a network is essentially a phylogenetic tree with some arcs added between the tree's edges so as to represent reticulate events such as gene transfer. Even so, this model does not permit the representation of evolutionary scenarios where reticulate events have taken place between different subfamilies or lineages of species. To represent such scenarios, in this paper  we introduce the notion of a {\em forest-based phylogenetic network}, that is, a collection of leaf-disjoint  phylogenetic trees on a set of species with arcs added between the edges of distinct trees within the collection. Forest-based networks include the recently introduced class of {\em overlaid species forests} which are used to model introgression.  As we shall see, even though the definition of forest-based networks is closely related to that of tree-based networks, they lead to new mathematical theory which complements that of tree-based networks. As well as studying the relationship of forest-based networks with other classes of phylogenetic networks, such as tree-child networks and universal tree-based networks, we present some characterizations of some special classes of forest-based networks. We expect that our results will be useful for developing new models and algorithms to understand reticulate evolution, such as gene transfer between collections of bacteria that live in different environments.


\end{abstract}

\section{Introduction}

In evolutionary biology, it is common to represent the evolution of a set of present-day species using a {\em phylogenetic tree}, that is a rooted, graph-theoretical tree whose leaves correspond to the species \cite{S16}. In recent years however, it has become increasingly recognized that phylogenetic trees may not provide an adequate means to represent the evolution of set of species in case the species or their ancestors have transferred or shared genetic material between one another in the past. This type of evolution is sometimes called {\em reticulate evolution}, and it includes evolutionary processes such as gene transfer between bacteria, hybridization between plants and recombination of viruses. Phylogenetic trees are not able to fully represent this type of evolution since they can only represent speciation or branching events (see e.g. \cite[Chapter 4]{huson2010phylogenetic}), and reticulate events require a graph where ancestors come together.

Despite this issue, phylogenetic trees can still be used as a starting point to represent reticulate evolution by, for example, taking some phylogenetic tree and 
then adding in extra edges to represent reticulate events (see e.g. \cite{makarenkov2001t}). We illustrate this in Figure~\ref{illustrate}(i), where 
we have started with a base-tree representing the evolution
of a hypothetical collection of bacteria, and added in 
some dashed arcs between arcs in the tree so as to represent past events 
where genes have been laterally transferred
between ancestral species (see e.g. \cite{kunin2005net}, \cite{makarenkov2021horizontal}
for  some real-world examples in bacteria and viruses, respectively). 
Mathematically speaking, the resulting 
graph theoretical structure is an example of a {\em phylogenetic
network}, that is, a rooted, directed acyclic graph with leaf-set 
corresponding to the present-day species (see e.g. \cite{S16}).
Note that directed cycles are not allowed in such networks 
since, for example, a species cannot be an ancestor of itself.
 
Phylogenetic networks that are created by adding in 
edges to a phylogenetic tree to form a network
are called {\em tree-based networks} \cite{FS15}. Since 
their formal introduction in \cite{FS15}, tree-based networks
have created a lot of interest in the literature. For example, it is known 
that not every phylogenetic network 
is tree-based \cite{vi-web}, and as a result several elegant
characterizations of tree-based networks have been developed 
(see e.g. \cite{FSS18}, \cite{FS15}, \cite{HS20} \cite{pons2019tree}, \cite{Z16}).
In addition, efficient algorithms have been presented for
deciding whether or not a phylogenetic network is tree-based (see e.g. \cite{FS15} and \cite{JvI18}).
There are also several results concerning the relationship
between tree-based networks and other special classes of phylogenetic 
networks, as well as structural results on 
spaces of tree-based networks (see e.g. \cite{fischer2020space}, \cite[Corollary 10.18]{S16}).
For a brief review of tree-based networks see \cite[Section 10.4.2]{S16}

\begin{figure}[h]
	\begin{center}
		\includegraphics[scale=0.6]{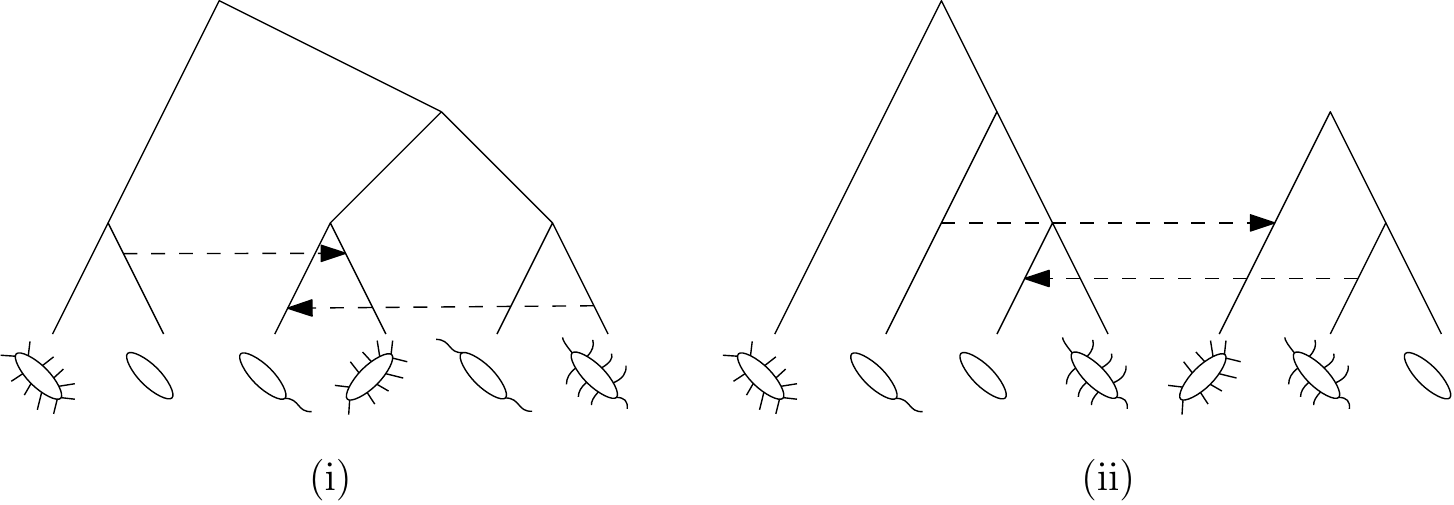}
		\caption{(i) A tree-based network and (ii) a forest-based network 
			for a collection of bacteria. The dashed arrows indicate a lateral
			gene transfer event between bacterial ancestors.}\label{illustrate}
	\end{center}
\end{figure}

Recently, {\em overlaid species forests} were introduced \cite{scholz2019osf}, 
which are closely related to tree-based networks (see also \cite{huber2020overlaid}). 
More specifically, in evolutionary studies it can be of interest to
understand how species within subfamilies of species (sometimes called
lineages or clades), swap genetic material between one another. 
For example, through the evolutionary process known as introgression
animals such as butterflies in one lineage can incorporate genes from other 
those in other lineages that give 
rise to new traits such as wing patterns \cite{wallbank2016evolutionary}
(see \cite{scholz2019osf} for more details concerning introgression).
To model this process, instead of adding  arcs to a
phylogenetic tree, it was proposed to 
add arcs between {\em different} trees within a collection of leaf-disjoint
phylogenetic trees or {\em phylogenetic forest}, for short. 
We call a network that results in this way a {\em forest-based network}. 
These networks generalize the notion of a phylogenetic 
network by permitting a network to have multiple roots. As we
shall see in Section~\ref{sec:forest-based},   forest-based networks 
also generalize overlaid species forests  since there are forest-based networks
that are not of this type.

To illustrate the concept of a forest-based network we return
to the hypothetical example of lateral gene transfer mentioned above. 
Suppose that each component of a phylogenetic forest 
is a phylogenetic tree for some collection of bacteria living in a certain
environment (say  in the human mouth or gut). Then 
the forest-based network in Figure~\ref{illustrate}(ii)
represents how the collections of bacteria living in different environments have swapped 
genes between one another in the past (whereas the tree-based network 
on the left represents lateral gene transfer within 
a single environment). In \cite{jeong2019horizontal}
such swaps are called inter-niche lateral gene transfer events.
Although at first sight, the concept of a forest-based network 
appears to be a relative simple modification of the definition of 
a tree-based network, in this paper we shall see that its
study requires the development of some interesting new theory.

We now summarize the contents of the rest of this paper.
In Section~\ref{sec:prelims}, we introduce basic terminology and notation. 
In Section~\ref{sec:forest-based}, we then present the
formal definition of forest-based networks and investigate some of their basic properties, 
for example, showing that every forest-based network has a 
special type of base forest  (Theorem~\ref{paths}).
In Section~\ref{sec:relationship}, we consider the relationship between 
forest-based {\em phylogenetic} networks (i.e. networks with a single root),
and other well-known classes of phylogenetic networks, including
tree-based networks and so-called tree-child networks \cite{cardona2008comparison}. 
In Sections~\ref{sec:arboreal} and~\ref{sec:cluster-arboreal}, we consider {\em arboreal networks} a
special class of forest-based networks whose underlying, undirected graph is a tree.
In particular, in Theorem~\ref{arboreal} we characterize arboreal 
networks that are forest-based, and in Theorem~\ref{lm-club} we show 
that two arboreal networks induce
the same set of clusters if and only if they are both forest-based.

In Section~\ref{sec:charaterize}, we consider the problem of characterizing 
forest-based networks. More specifically, in Theorem~\ref{theo:general}
we characterize proper forest-based networks, that is,  forest-based networks 
with $m\ge 2$ roots which are based on a phylogenetic forest which has $m$ components.
We also show that there is a simple characterization for forest-based networks in case $m=2$ which 
can be given in terms of the existence of a 2-coloring of a certain graph that can be 
associated to any network (Theorem~\ref{col2}).
This is somewhat similar to the characterization of binary tree-based networks given in \cite{JvI18}.
In Section~\ref{sec:universal}, we then turn our attention to 
the concept of universal forest-based networks, that is 
networks that contain all possible phylogenetic forests as a base forest.
These are a natural generalization of universal tree-based networks, which
contain every possible phylogenetic tree as a base tree.  
Although universal tree-based networks always exist \cite{H16,Z16}, 
in Section~\ref{sec:universal} we show there
are no universal forest-based networks with four or more leaves (Theorem~\ref{no-universal}). 
In Section~\ref{sec:conclusion} we conclude by 
presenting some potential directions for future work.

\section{Preliminaries}
\label{sec:prelims}

Throughout this paper, we assume that $X$ is a non-empty, finite set, which
can be thought of as a collection of species.

We shall use standard terminology from graph theory (see e.g. \cite[Section 1.2]{S16}).
Unless stated otherwise, we assume 
that all graphs are directed and that they have no parallel arcs or loops. 
Suppose $G$ is a graph.  We denote the vertex 
set of $G$ by $V(G)$ and its set of arcs by $A(G)$. 
Suppose $u,v\in V(G)$.  We denote an arc $a$ from $u$
to $v$ by $a=(u,v)$, and refer to
$u$ and $v$ as the {\em end vertices} of $a$
and $v$ and $u$ as the {\em head} and {\em tail} of $a$, respectively.
We say that $v$ lies \emph{below} $u$
if there exists a directed path in $N$ from $u$ to $v$ (so, in particular,
$v$ is below $v$).  If, in addition, $v\not=u$ then we say that $v$ 
lies {\em strictly below} $u$. We call $u$ an \emph{ancestor} 
of $v$ if $v$ is below $u$.
If $u$ and $v$ are such that $(u,v)$ is an arc of $N$, 
then we call $u$ a \emph{parent} of $v$ and $v$ a \emph{child} of $u$.

For $v\in V(G)$, we refer to the number of arcs with head $v$ as the
{\em indegree of $v$}, denoted by $indeg(v)$,
and to the number of arcs with tail $v$ the {\em outdegree
of $v$}, denoted by $outdeg(v)$. We call $v$ a {\em leaf} 
of $G$ if $indeg(v)=1$ and $outdeg(v)=0$, unless
$V(G)=\{v\}$ in which case we also call $v$ a leaf. 
We denote by $L(G)$ the set of 
all leaves of $G$. In case $v$ is not a leaf
of $G$ we call $v$ an {\em internal vertex} of $G$. 
If $indeg(v)= 1$, then we refer to $v$ as a 
{\em tree vertex} of $G$, and if $indeg(v)=0$, 
then we call $v$ a {\em root} of $G$. 
We call every internal vertex
of $G$ that is neither a root nor a tree-vertex
a {\em hybrid vertex} of $G$.
The set of all roots of $G$ is denoted by $R(G)$
and the set of all hybrid vertices of $G$ is by $H(G)$.
We say that $G$ is {\em semi-binary} if $G$ consists of a single vertex or 
if every hybrid vertex of $G$ has indegree two and outdegree one, and $G$ is {\em binary}
if, in addition to being semi-binary, every root and
every non-leaf tree vertex has outdegree two.

We say that $G$ is {\em acyclic} if it contains
no directed cycles, and call $G$ a {\em tree} if it 
has a single root, all arcs in $G$ are 
directed away from the root, and the underlying, undirected 
graph of $G$ is a tree (note that we regard a vertex as being a tree).  
We call $G$ a {\em forest} if it has at least two connected 
components and all of its connected components 
are trees. For convenience, we will 
sometimes also regard a forest as being the set of trees which
make up its components. 

A {\em multiply rooted phylogenetic network $N$ (on $X$)} or 
{\em network (on $X$)}, is a  semi-binary, connected, acyclic graph 
with leaf set $X$ and at least one root, in 
which every root in $R(N)$ has outdegree 
at least 2. In case the number $m=|R(N)|$ of roots
in a network $N$ is of relevance to the
discussion, we sometimes also call $N$ an \emph{$m$-network (on $X$)}.
If $N,N'$  are networks on $X$, then 
we say that $N$ and $N'$ are {\em equivalent} if there exists 
a bijective map $\psi:V(N)\to V(N')$ that induces a graph isomorphism 
between $N$ and $N'$ and that is the identity on $X$.
If $|R(N)|=1$, then $N$ is called a \emph{phylogenetic 
network (on $X$)}. In this case, we denote the root of $N$ by $\rho(N)$. If $N$ is such that $H(N)$ is empty, 
then we call $N$ a \emph{phylogenetic tree (on $X$)}.
Note that in the special case where $X=\{x\}$, we regard the graph
with the single vertex $x$ as a phylogenetic tree on $X$
with leaf and root vertex $x$. 
A {\em phylogenetic forest $F$ (on $X$)} is a set consisting of at least two
phylogenetic trees so that $L(T) \cap L(T') =\emptyset$ for all $T,T' \in F$, and 
$\bigcup_{T\in F} L(T) =X$.

We conclude this section by introducing two operations on a graph. Suppose that $G$ is a graph 
and that  $a=(u,v)$ is an arc of $G$. Then we refer to the process of deleting $a$, adding 
a new vertex $w$, and adding arcs $(u,w)$ and $(w,v)$ as {\em subdividing} $a$.
In this case, we also refer to $w$ as a {\em subdivision vertex} of $a$. We call 
a graph $G'$ a \emph{subdivision} of $G$ if $G'$ is 
isomorphic to a graph that can be 
obtained from $G$ via a finite sequence 
of subdivisions. Furthermore,
we refer to the process that reverses subdivision (i.e. for a 
vertex $v\in V(G)$ with indegree and outdegree one, delete $v$ 
and its incoming and outgoing arcs and add a new arc from the parent of $v$ to the child of $v$) 
as {\em suppressing of $v$}. We also refer to the process that
removes a root $\rho$ with outdegree 1 in a graph
and the arc with tail $\rho$  as suppression.

\section{Forest-based networks}
\label{sec:forest-based}

In this section, we formally define forest-based networks and 
present two basic results concerning their structure.
Note that the concepts that we use to define a forest-based network 
are closely related to the ones used to define a tree-based network in \cite[p.257]{S16}.

We define a network $N=(V,A)$ on $X$ to be {\em forest-based} if
there exists a subset $A' \subseteq A$ such that $F'=(V,A')$ is a
forest with the same leaf set as $N$, and so that 
every arc in $A-A'$ has end vertices contained in different trees of $F'$. 
Note that this implies $|X| \ge 2$.
We call $F'$ a {\em subdivision forest} for $N$,
the arcs in $A-A'$ {\em contact arcs} and the vertices in $F'$
with indegree and outdegree both equal to one {\em subdivision vertices (of $F$)}.
We call the phylogenetic forest $F$ on $X$ that we obtain
by repeatedly suppressing all subdivision vertices and outdegree one roots in
each component of $F'$ until we obtain a phylogenetic tree 
a {\em base forest} for $N$. We also say that $N$ is {\em based on $F$}, and that the
forest $F'$ provides an {\em embedding} of $F$ into $N$.
Note that $X=L(F)=L(F')$, and that in case a component $C$ of $F$ consists of
a single element, then the component of $F'$ which
gives rise to $C$ is necessarily a path. For $m\geq 2$, 
we call an $m$-network $N$  {\em proper} forest-based if it 
contains a {\em proper} base forest, that is, a  base forest
with $m$ roots. See Figure~\ref{change} for illustrations of these
concepts.
	
\begin{figure}[h]
	\begin{center}
		\includegraphics[scale=0.6]{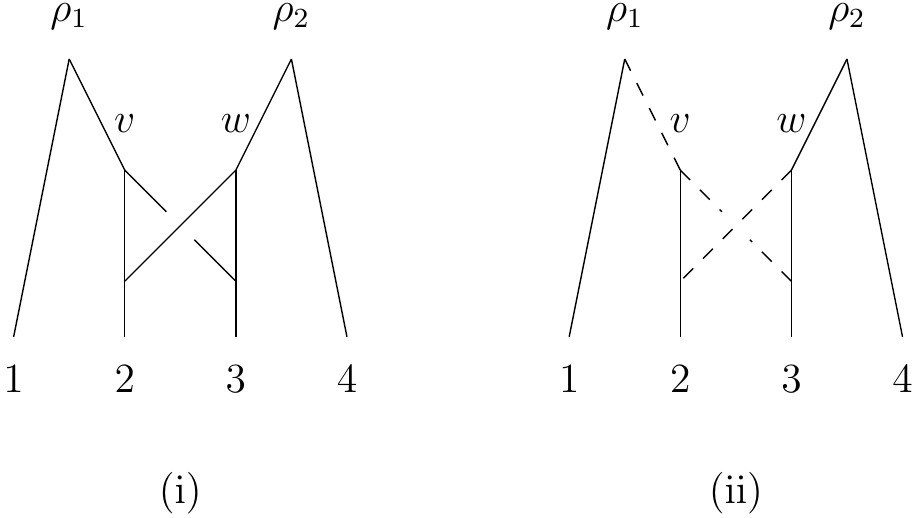}
		\caption{(i) A forest-based network $N$ on the set $X=\{1,2,3,4\}$.	
		For example, it is based on the phylogenetic 
			forest $F$ consisting of the isolated vertices labelled $1$ and $2$ and the 
			phylogenetic tree with a single root on $\{3,4\}$. 
			The network is also proper forest-based since it has the proper 
			base forest consisting of the two phylogenetic trees on $\{1,2\}$ and $\{3,4\}$. 
			(ii) An embedding of the forest $F$ into $N$, where
			the contact arcs are indicated as dashed arcs. \label{change}}
	\end{center}
\end{figure}
	
Before proceeding we note that not all forest-based networks
are overlaid species forests, and so the concept of
a forest-based network is more general than that considered in \cite{scholz2019osf}
(the formal definition of an overlaid species forest is quite involved and so
we shall not present it here). For example, the forest-based 
network $N$ in Figure~\ref{change}(i) is not an overlaid species forest.
To see this, note that for every embedding of some base forest into $N$, 
one of the arcs with tail $v$ and one of the arcs with tail $w$ must be a contact arc. 
However, one of the conditions for a network to be 
an overlaid species forest is that all contact arcs must share 
an ancestor (cf. \cite[Theorem 5.3]{huber2020overlaid}), which 
is not possible for any pair of contact arcs that have tail $v$ and tail $w$. 
In the next section we shall present some further examples and results which
will elucidate the relationship between forest-based networks
and various other classes of networks.  

We now show that a forest-based network can be thought of as 
a phylogenetic forest with some arcs added in between different components of the forest
(this is analogous to \cite[Proposition 10.16]{S16} for tree-based networks).  

\begin{lemma}
	Suppose $N=(V,A)$ is a network on $X$, $|X|\geq 2$. Then $N$ is forest-based if and only if there is a set 
	$I \subseteq A$ such that $F'=(V,A-I)$ is a forest, every arc in $I$ has
	its end vertices in different trees of $F'$, and for every non-leaf vertex $v$ of $N$,
	there exists an arc with tail $v$ that is not contained in $I$. In particular, 
	if $N$ is {\em binary}, then $N$ is forest-based if and only if there is a set $I \subseteq A$ 
	such that $F'=(V,A-I)$ is a forest,  every arc in $I$ has
	its end vertices in different trees of $F'$, and for every pair of distinct arcs in $ I$ with
	a vertex $v$ in common, $v$ is a root of $F'$ that is not a component of $F'$.
\end{lemma}

\begin{proof}
	Suppose that $N$ is forest-based. Let $I = A- A'$ for $F' = (V, A')$ some subdivision
	forest for $N$. Then $F'$ is clearly a forest. Suppose $v$ is a non-leaf vertex
	of $N$. Since $F'$ is a subdivision 
	forest for $N$, we have that $L(F')=L(N)$. Hence,
	$v$ is not a leaf of $F'$. Thus, there is an arc with tail $v$ that
	belongs to $A'$.
	
	Conversely, suppose that $I$ is as in the statement of the lemma 
	so that $F' = (V,A-I)$ is a forest. Then clearly $L(N) \subseteq L(F' )$. Now, suppose that $v$
	is a non-leaf vertex of $N$. By assumption, there exists an arc with tail $v$ that
	is not in $I$. In particular, $v$ is not a leaf of $F'$, and so $L(F')=L(N)$.	
\end{proof}

Note that, as we have seen in Figure~\ref{change}, a forest-based network
might have more than one base forest and
different base forests for the network do not necessarily need to have the 
same  number of components. We conclude this 
section by showing that every forest-based network must have a special 
type of base forest with $|X|$ components. For $X$ with $|X| \ge 2$, we define the {\em trivial (phylogenetic) forest}
on $X$ to be the phylogenetic forest in which every component is a vertex (i.e an element of $X$).

\begin{theorem}\label{paths}
	Suppose that $N$ is a network on $X$. Then the following are equivalent 
	\begin{itemize}
	\item[(i)] $N$ is forest-based. 
	\item[(ii)]  $N$ is based on the trivial forest. 
	\item[(iii)] The trivial forest is embedded in $N$ as an union of paths (some possibly of
	length 0), and there is no arc in $N$ joining two non-consecutive vertices of the same path.
	\end{itemize}
\end{theorem}

\begin{proof}
Clearly, (ii) implies (i), and (ii) and (iii) are equivalent.

To show that (i) implies (iii), suppose (i) holds and that $N$ is forest-based with
base forest $F$. Let $F'$ be an embedding of $F$ such that $F'$ is not a union of paths. Then there exists
a component $C'$ of $F'$ that is not a path. Hence, $C'$ contains a vertex $v$ that has  outdegree greater than one, and no ancestor of $v$ in $C'$ has outdegree greater than one. 
By removing from $C'$ all but one arc with tail $v$, we obtain an embedding $F''$
of a new base forest, such that the number of vertices of outdegree two or more in $F''$ is
strictly lower than the number of vertices of outdegree two or more in $F'$. We can then repeat
this process until we obtain an embedding 
of a forest $F_0$ such that all vertices in $F_0$ have outdegree
at most one. So $F_0$ is a union of paths 
such that there is no arc in $N$ joining two non-consecutive vertices of the same path, and therefore 
an embedding of the trivial forest in $N$. So (iii) holds.
\end{proof}

\begin{corollary}
Suppose that $N$ is an $m$-network on $X$.  If $N$ is forest-based, then $|X| \geq m$.
Moreover, if $|X|>m$, then $N$ must contain a base forest that is not 
proper, and if $|X|=m$, then $N$ must be proper.
\end{corollary}
\begin{proof}
Suppose that $N$ is forest-based. Then 
any base forest for $N$ contains at least $m$ phylogenetic trees (since each root of $N$ must belong to a 
different tree), and each of these trees has at least one leaf. So if $F$ is a base forest for $N$, then
$|X| \geq |F| \geq m$. The last statement now follows immediately by Theorem~\ref{paths}.
\end{proof}

\section{Relationship of forest-based networks with other classes of networks}
\label{sec:relationship}

We now present some results and examples to
elucidate the relationship between forest-based networks
and some well-known classes of phylogenetic networks. 
We shall focus on binary networks, as
binary phylogenetic networks are most commonly studied in the literature.

We begin by noting that there are networks that are 
neither tree-based nor forest-based (see e.g. \cite[Figure 10.10(c)]{S16}).
Thus it is of interest to better understand the relationship
between binary forest-based networks, tree-based networks and
other classes of networks. More specifically, we shall consider 
so-called tree-child,  tree-sibling and reticulate-visible networks (see below for definitions)
since, in case these have a single root, they are well-understood classes that have interesting interrelationships with 
tree-based networks \cite[Figure 10.12]{S16}. 

We begin by showing that binary forest-based {\em phylogenetic} networks
are always tree-based.

\begin{proposition} \label{prop:forest-tree}
	Suppose that $N$ is a binary phylogenetic network on $X$, $|X|\ge 2$. If $N$ is 
	forest-based, then it is tree-based. 
\end{proposition}
\begin{proof}
	Assume that $N$ is forest-based with base forest $F$, and consider the embedding $F'$ of $F$
	into $N$. For all trees $T'$ of $F'$ whose root $\rho_{T'}$ is distinct from the root of $N$,
	we can add to $F'$ the incoming arc of $\rho_{T'}$ in $N$ (choosing one such arc if $\rho_{T'}$ is a hybrid vertex of $N$). Clearly, the embedding $F_0$ obtained this way is the embedding of a base tree for $N$. In particular, this means that $N$
	is forest-based.
\end{proof}

We now consider tree-child networks. Generalising the definition for tree-child 
phylogenetic networks  \cite{cardona2008comparison}, for a network $N$ on $X$, we
define $N$ to be  \emph{tree-child} if all internal vertices of $N$ have a child that is a tree vertex.

Note that any binary tree-child phylogenetic network is tree-based (see e.g. \cite[Corollary 10.18]{S16}). 
We now show that a similar result holds for binary forest-based networks.

\begin{theorem}\label{tc}
	If $N$ is a binary tree-child network on $X$, $|X|\ge 2$, then it is forest-based.
\end{theorem}
\begin{proof}
	Let $I \subseteq A$ be the set of arcs $(u,v)$ in $N$ such that $v$ is a hybrid vertex of
	$N$. 
	If $I=\emptyset$ then $N$ is a phylogenetic tree. Since a phylogenetic 
	tree on $X$ is based on the trivial phylogenetic forest, the theorem follows. 
	So assume that $I\not=\emptyset$. Clearly, the graph $F'=(V,A-I)$ is a forest. We remark first that a leaf of $F'$ is
	either a leaf of $N$, or a vertex of $N$ whose children are all hybrid vertices. The
	network $N$ being tree-child, it contains no vertices of the latter type, so
	$L(F')=L(N)$. Moreover, all arcs $(u,v)$ of $I$ are such that $v$ is a root of  $F'$, so
	in particular $u$ and $v$ belong to distinct trees of $F'$. This means that $F'$ is a
	subdivision forest for $N$, so $N$ is forest-based.
\end{proof}

In particular, it follows from Theorem~\ref{tc} that all binary
tree-child {\em phylogenetic} networks (including phylogenetic trees!) are forest-based. 

We now consider two further classes of networks. 
A network is  (1) {\em tree-sibling} if for every $v \in H(N)$
there is a $v' \in V(N)$ so that $v'$ is a tree-vertex and $v'$ shares a parent with $v$, 
and (2) \emph{reticulation-visible} if for every $v \in H(N)$, 
there is a leaf $x \in X$ such that all directed paths from a root of $N$ to $x$ contain $v$. 
These definitions generalize the 
ones that were originally given for phylogenetic networks 
(see \cite{nakhleh2004phylogenetic} and \cite{huson2007beyond}, respectively).
Note that it follows immediately from the definitions that 
tree-child networks are  tree-sibling and reticulation-visible.

In Figures~\ref{classes} and \ref{examples-1} we present 
a diagram and some examples which illustrate 
the interrelationship between forest-based 
phylogenetic networks and the other classes of phylogenetic networks
that we have considered (see also \cite[Figure 10.12]{S16}).
Note that the network $G$ in Figure~\ref{examples-1} provides an example which shows
that we do not necessarily obtain a forest-based network by removing the root from a tree-based phylogenetic network.

\begin{figure}[h]
	\begin{center}
		\includegraphics[scale=0.55]{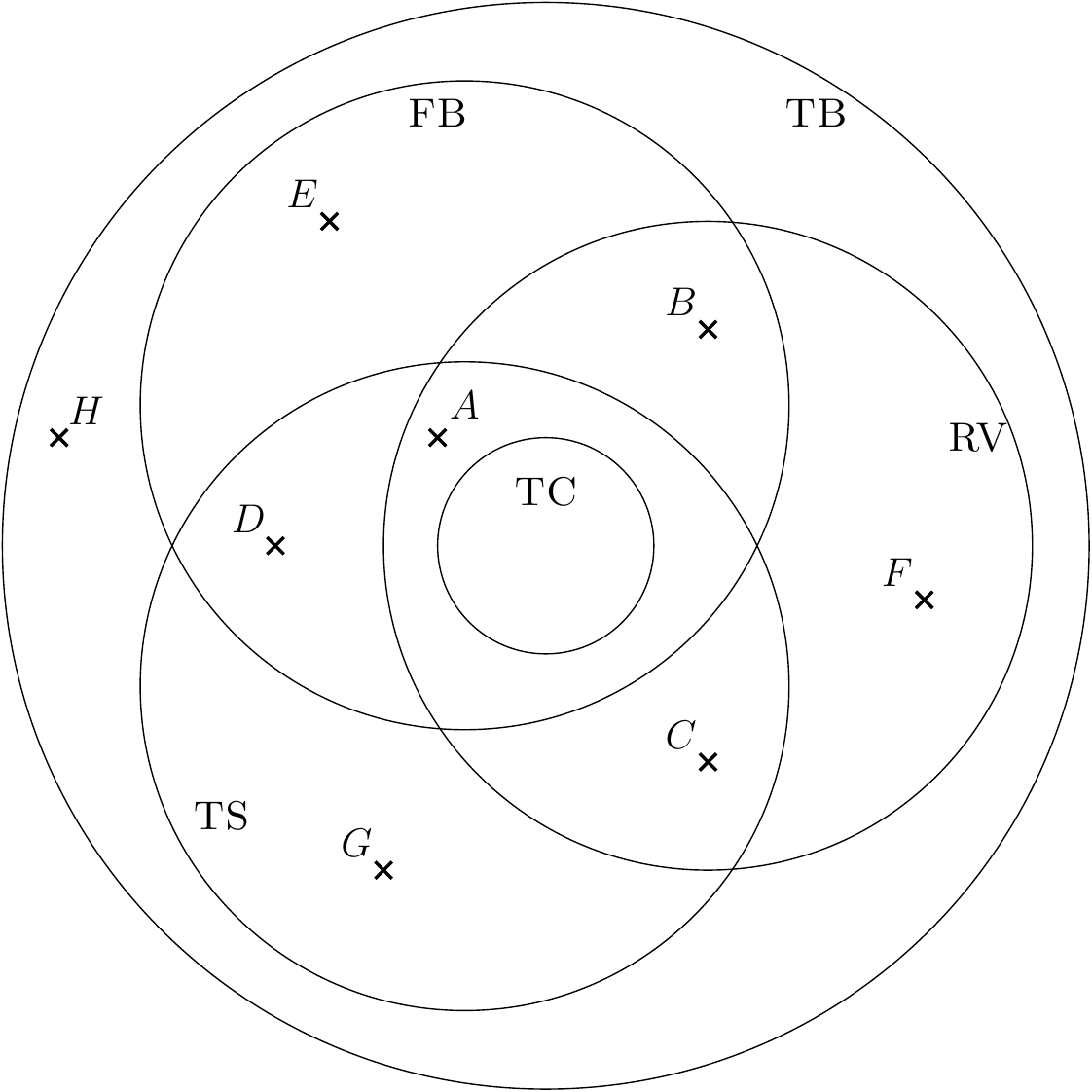}
		\caption{A Venn-diagram for different classes of phylogenetic networks; 
			$TC=\mbox{tree-child}$, $FB=\mbox{forest-based}$,  $TB=\mbox{tree-based}$,  $TS=\mbox{tree-sibling}$, 
			and $RV=\mbox{reticulate-visible}$. See Figure~\ref{examples-1}
			for the indicated networks  $A$--$H$.\label{classes}}
	\end{center}
\end{figure}
 
\begin{figure}[h]
	\begin{center}
		\includegraphics[scale=0.55]{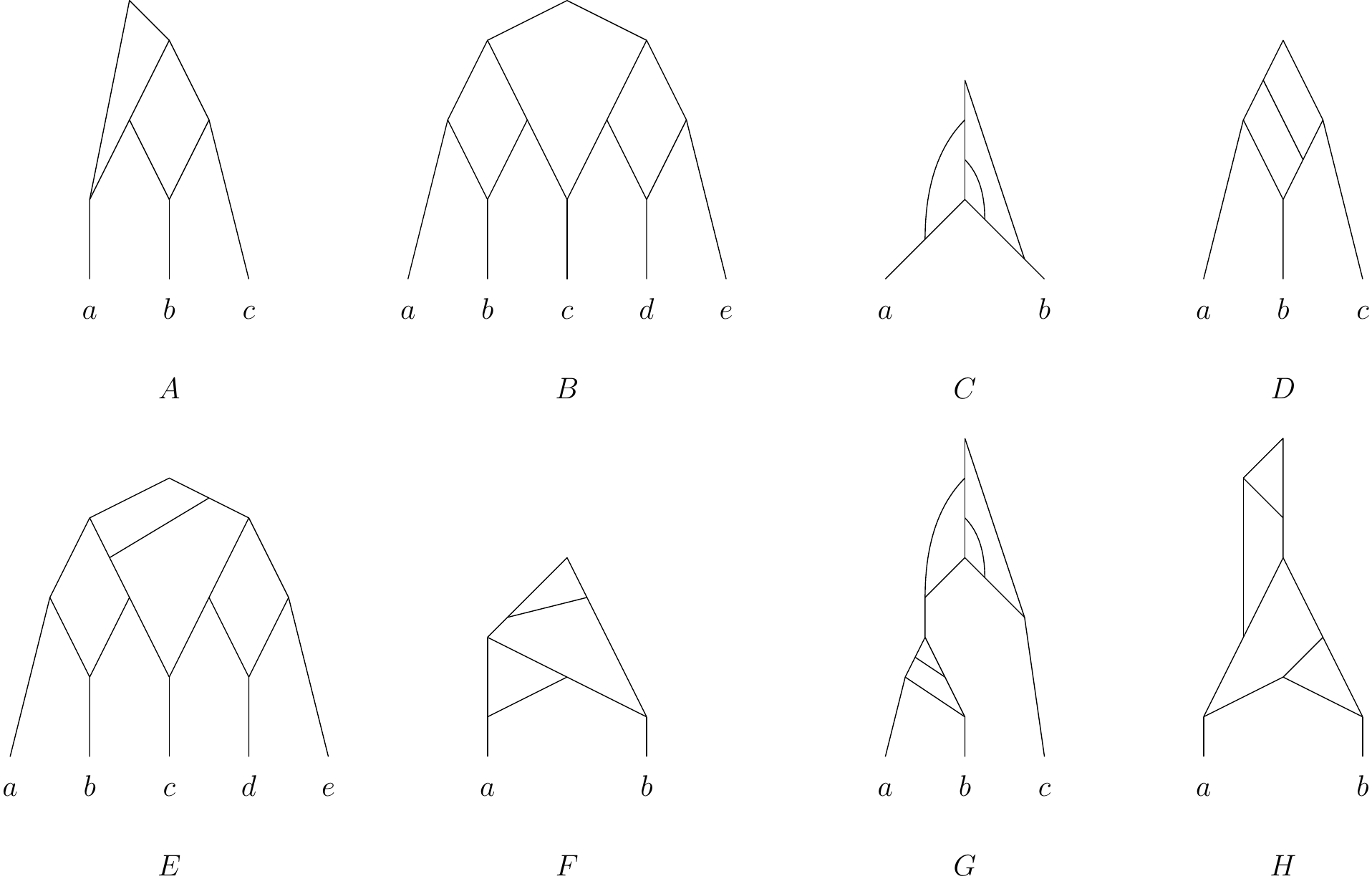}
		\caption{Eight examples of phylogenetic networks.\label{examples-1}}
	\end{center}
\end{figure}

\section{Arboreal networks}
\label{sec:arboreal}

An {\em arboreal} network is a network whose underlying (undirected) graph is a tree.
These networks are of interest as, even though
an arboreal network with more than one root has an
underlying tree structure, it must still contain some reticulation vertices
(since, as can be easily seen, if $N$ is arboreal, then $|H(N)|=|R(N)|-1$).

In this and the next section we shall consider properties of
forest-based arboreal networks.  Note 
that arboreal networks are not necessarily forest-based (see e.\,g.\,Figure~\ref{arb-nfb}).
In this section, we shall prove the following  
characterization for when an arboreal network is forest-based.

\begin{figure}[h]
	\begin{center}
		\includegraphics[scale=0.6]{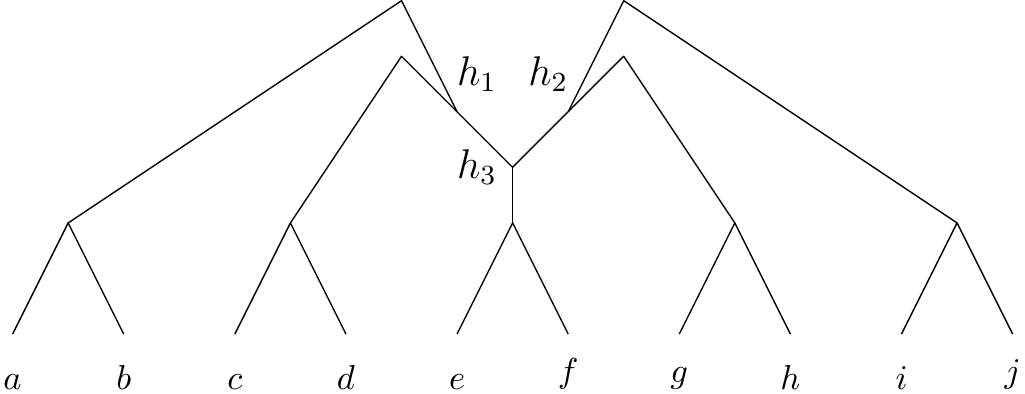}
		\caption{An arboreal network $N$ that is not forest-based. 
			To see this, note that one of $(h_1,h_3)$ or $(h_2,h_3)$ must be a 
			contact arc, in which case either $h_1$ or $h_2$ becomes 
			a leaf in the corresponding forest.} \label{arb-nfb}
	\end{center}
\end{figure}

\begin{theorem}\label{arboreal}
Suppose that $N$ is an arboreal network with two or more leaves. 
Then $N$ is forest-based if and only if for all hybrid vertices $h$ of $N$, 
there is a sequence of distinct vertices 
$v_1=h, \ldots, v_k$, $k\ge 1$,
such that any two consecutive vertices in the sequence share a child that is a hybrid vertex, and 
$v_k$ has a child that is not a hybrid vertex.
Moreover, in case this holds, then $N$ is proper forest-based if and only if $|R(N)| \ge 2$.
\end{theorem}

\begin{proof}
First note that if $N$ has a single root, then $|H(N)|=|R(N)|-1=0$, and so 
$N$ is a phylogenetic tree. Thus, it is forest-based by Theorem~\ref{tc}. 
Moreover, in  this case $N$ is clearly not proper forest-based.
So, we shall assume from now on that $|R(N)| \ge 2$.
 
Suppose first that $N$ is forest-based, with 
subdivision forest $F'$. Let $h$ be a hybrid vertex of $N$. 
We shall associate a sequence of vertices $\sigma(h)$ to $h$ and show that
this sequence satisfies the properties stated in the theorem.
If the child of $h$ is not a hybrid vertex, then 
$\sigma(h)$ is the sequence that contains $h$ as its sole vertex. 
Clearly, $\sigma(h)$ satisfies the properties stated in the theorem.
So assume that the child of $h$ is a hybrid vertex. Let
$i\geq 1$ be such that for all $1\leq j\leq i$ we have that
the child of $v_j$ is a hybrid vertex. Then we
define $v_{i+1}$ from $v_i$ as follows. First we pick a child $w_i$ of $v_i$ such that $(v_i,w_i)$ 
is an arc of $F'$. Note that such a child must exist since $F'$ is a subdivision forest 
for $N$. Then we choose $v_{i+1}$ to be a parent of $w_i$ that is not $v_i$. 
Note that the choice of $w_i$ implies that the arc $(v_{i+1},w_i)$ is not an arc of $F'$. 
In particular, we cannot have $v_{i+1}=v_{i-1}$ by the choice of $i$. Since $N$ is 
arboreal, it follows that a given vertex of $N$ cannot appear twice in $\sigma(h)$. 
As the number of vertices in $N$ is finite, it follows that $\sigma(h)$ must end 
in a vertex $v_k$ that has a tree vertex as a child.

Conversely, suppose that for all hybrid vertices $h\in H(N)$, there exists a 
sequence $\sigma(h)$ of vertices that satisfies the stated properties. We next 
construct a set $I$ of arcs of $N=(V,A)$ such that each arc of $I$ 
has a hybrid vertex of $N$ as head, and the graph $F'=(V,A-I)$ satisfies $L(F')=L(N)$. 
To do this, we start by constructing a graph $\mu(N)$ as follows: The 
vertices of $\mu(N)$ are all vertices of $N$ with at least one child that is a hybrid 
vertex, and two vertices of $\mu(N)$ are joined by an edge if they share a child. We 
first remark that since $N$ is arboreal, $\mu(N)$ does not contain cycles. We also 
remark that there is a trivial bijection $\chi$ between 
the edge set of $\mu(N)$ and $H(N)$. 

We next orient the edges of $\mu(N)$ to obtain a directed graph $\mu^+(N)$ 
that has the same vertex set as $\mu(N)$. For this it suffices to consider a
connected component $G$ of $\mu(N)$.
To this end, note that $G$ is an unrooted tree and that 
a vertex of $G$ with overall degree one is either a hybrid vertex of $N$, 
or a vertex of $N$ with at least one child in $N$ that is a tree-vertex. We start by 
successively considering the vertices of $G$ corresponding to hybrid vertices of $N$
under $\chi$. Let $h$ be such a vertex of $N$. Then, by assumption, 
sequence $\sigma(h) = (v_1=h, v_2,\ldots, v_k)$, $k\geq 1$, satisfies the 
properties stated in the theorem.
For $i \in \{1, \ldots, k-1\}$ and all $j\in \{1,\ldots, i-1\}$, assume that the
edges $\{v_j,v_{j+1}\}$ have already been oriented,  and that the
edge $e=\{v_i,v_{i+1}\}$ has not yet been assigned an 
orientation. Then we direct $e$ from $v_i$ to $v_{i+1}$.

Once all vertices on $\sigma(h)$ have been processed, we 
orient all edges of $G$ whose tail is a vertex $v$ in $\sigma(h)$ and 
which have not already been processed away from $v$. Repeating this 
process for all vertices that are heads in the resulting graph and so on,
results in an  oriented graph $\mu^+(N)$. Note that this includes the case 
where $G$ does not contain any vertex corresponding to a hybrid vertex of $N$. 
By construction, our assumptions on $\sigma(h)$ imply that a
vertex of $G$ of outdegree $0$ has a child in $N$ that is a tree-vertex, 
as desired. Furthermore, $\chi$ induces a natural 
bijection $\chi^+$ between the arc set of $\mu^+(N)$ and $H(N)$.

Armed with $\mu^+(N)$, we construct a set $I$ of arcs of $N$ 
as follows. First, we initialize $I$ with the empty set. Next, for each 
hybrid vertex $h$ of $N$, we add the arc $(v,h)\in A$ to $I$, where $v$ 
is the head of the arc in $\mu^+(N)$ corresponding to $h$ under $\chi^+$.

Clearly, the graph $F'=(V,A-I)$ is a forest, since it contains exactly one 
incoming arc for each hybrid vertex of $N$. To see that $L(F')=L(N)$, it 
suffices to remark that each non-leaf vertex $v$ of $N$ has an outgoing 
arc in $F'$. If $v$ has at least one child that is not a hybrid vertex, then 
the set equality holds. Otherwise, $v$ is a vertex of $\mu(N)$ 
whose indegree in $\mu^+(N)$ is at least one. By definition of $I$, for $h$ 
a hybrid vertex corresponding to an outgoing arc of $v$ 
in $\mu^+(N)$ under $\chi^+$, the arc $(v,h)$ is an arc of $F'$.

To conclude that $F'$ is a subdivision forest for $N$, it suffices 
to remark that since $N$ is arboreal, there exists no arc in $I$ whose 
both end vertices are in the same tree of $F'$. So $N$ is forest-based.
Moreover, we have that $|R(N)|=|R(F')|$. Thus, $N$ is proper forest-based.
\end{proof}

\section{Cluster systems from arboreal, forest-based networks}
\label{sec:cluster-arboreal}

In phylogenetics, it is common to work with rooted phylogenetic trees
 in terms of clusters that they induce 
	as these can be sometimes easier to handle
	(e.g. for consensus methods  or for computing distances between phylogenetic trees --
	cf. e.g. \cite[Section 2.2.2]{S16}).
We can also associate clusters to networks as follows.
Suppose $N$ is a network on $X$ and $u\in V(N)$. We 
call the set $C(u)=C_N(u)$ of leaves of $N$ below  
$u$ the \emph{cluster} induced by $u$. If $|C_N(u)|=1$ then
we call $C_N(u)$ a {\em trivial cluster (on $X$)}. We refer 
to the set $\mathcal C(N)$ of all clusters induced by 
the vertices in $V(N)$ as the {\em cluster system} induced by $N$ 
and, more generally, we also refer to any collection of 
non-empty subsets or clusters in $X$ by the same name.

\begin{figure}[h]
	\begin{center}
		\includegraphics[scale=0.6]{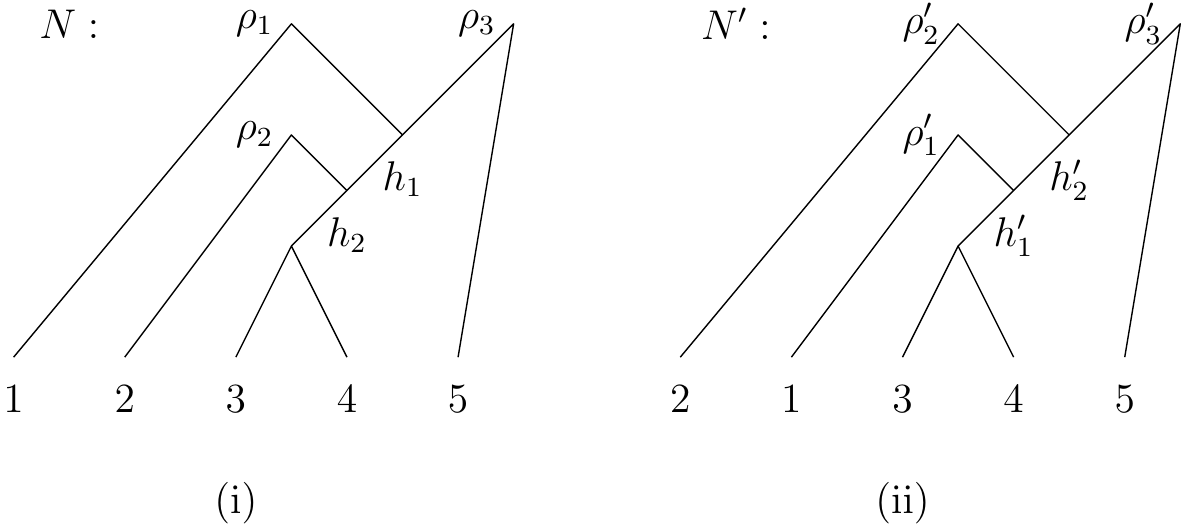}
		\caption{Two arboreal, 3-rooted networks $N$ and $N'$.
			Note that $\mathcal C(N)=\mathcal C(N')$, but that 
			$N$ and $N'$ are not equivalent. The bad 
			arcs are the arcs $(h_1,h_2)$ and $(h_2',h_1')$.  } \label{inequivalent}
	\end{center}
\end{figure}

Interestingly, in contrast to phylogenetic trees, there are inequivalent arboreal networks 
that have the same cluster systems  (see e.\,g.\,Figure~\ref{inequivalent}).
Even so, in this section we shall show that if $N$ and $N'$
are distinct arboreal networks with $\mathcal C(N)=\mathcal C(N')$, then $N$ is forest-based if and only if $N'$ is forest-based  (Theorem~\ref{lm-club}).  To do this, we 
will first prove two equivalence results 
for arboreal networks (Theorem~\ref{ts-clu} and Theorem~\ref{tu-clu}) that are analogous to the well-known equivalence theorem
between phylogenetic trees and hierarchies.
This latter result states that, given a cluster system $\mathcal C$, there is a phylogenetic 
tree $T$ on $X$ such that $\mathcal C(T)=\mathcal C$ 
if and only if $\mathcal C$ is a \emph{hierarchy (on $ X$)} (that is, $\mathcal C$ contains all 
trivial clusters and $X$ and, 
for all $C, C' \in \mathcal C$, $C \cap C' \in \{C,C',\emptyset\}$) and that, 
if such a phylogenetic tree $T$ exists, then
up to equivalence, $T$ is uniquely determined by $\mathcal C(T)$ 
(see e.\,g.\,\cite[Proposition 2.1]{S16}). 
 
To state our first result we require further definitions.
We say that an arboreal network $N$ is
{\em uniquely determined} by $\mathcal C(N)$ if 
any arboreal network $N'$ for which $\mathcal C(N)=\mathcal C(N')$ 
holds is equivalent to $N$. 
Furthermore, for $v\in R(N)$, 
we denote by $T(v)$ the subtree of $N$ spanned by all 
leaves below $v$, and we denote by $T_v$ the phylogenetic 
tree obtained from $T(v)$ by suppressing all vertices $u$ 
with $indeg(v)= 1= outdeg(u)$.
Given a cluster system $\mathcal C$ on $X$, 
we denote by $\mathcal I(\mathcal C)$ the 
graph whose vertex set is $\mathcal C$ and whose 
edge set is the set of pairs 
$\{C,C'\}\in {\mathcal C\choose 2}$ such 
that $C \cap C' \neq \emptyset$, and by 
$\mathcal C_M\subseteq \mathcal C$ the collection of 
 set-inclusion maximal 
elements of $\mathcal C$.

\begin{theorem}\label{ts-clu}
	Let $\mathcal C$ be a cluster system on $X$. Then there exists an arboreal  $| \mathcal C_M|$-rooted 
	network $N$ such that $\mathcal C(N)=\mathcal C$ if and only if:
	\begin{itemize}
		\item[(P1)] For all $C \in \mathcal C_M$, the set $\{C' \in \mathcal C\,:\, C' \subseteq C\}$ 
		is a hierarchy that contains all trivial  clusters on $C$.
		\item[(P2)] The graph $\mathcal I(\mathcal C_M)$ is connected. 
		\item[(P3)] For any two $C_1,C_2 \in \mathcal C_M$, we have $C_1 \cap C_2 \in \mathcal C \cup \{\emptyset\}$.
	\end{itemize}
\end{theorem}

To establish this result, we will use the following lemma:

\begin{lemma}\label{max-c}
	Let $N$ be an arboreal network on $X$. Then the set inclusion 
	maximal elements of $\mathcal C(N)$ are precisely the clusters $C(r)$ with $r \in R(N)$.
\end{lemma}

\begin{proof}
	If $|X|=1$ then the lemma trivially holds. So assume that $|X|\geq 2$.
	Clearly, all set-inclusion maximal elements $C$ of $\mathcal C(N)$ are 
	such that $C=C(r)$, for some $r \in R(N)$. Assume for contradiction that 
	there exists a root $r \in R(N)$ such that $C(r)$ is not set-inclusion maximal 
	in $\mathcal C(N)$. Then there must exists $v \in V(N)$ such that $C(r) \subsetneq C(v)$.  
	Hence, for all $x\in C(r)$, there exists a directed path from $v$ to $x$. 
	Since $v$ cannot be an ancestor of $r$ (as $r$ is a root of $N$), it follows 
	that $v$ and $r$ are vertices in a cycle in the underlying undirected graph 
	of $N$ which contradicts the assumption that $N$ is arboreal.
\end{proof}

\noindent{\em Proof of Theorem~\ref{ts-clu}:}
	Since the theorem clearly holds if $|X|=1$, we may assume that $|X|\geq 2$.
	Put $m=| \mathcal C_M|$.
	Assume first that there exists an arboreal $m$-network $N$ 
	such that $\mathcal C(N)=\mathcal C$. By Lemma~\ref{max-c},  
	we have $\mathcal C_M=\{C(r)\,|\, r\in R(N)\}$.  Let $C \in \mathcal C_M$ and 
	let $r$ be the root of $N$ such that $C=C(r)$. 
	Then, since $N$ is arboreal, $T_r$ is a phylogenetic tree on some subset $X_r$ of $X$.
	Hence, $\mathcal C(T_r)$ is a hierarchy on $X_r$. Since $\mathcal C(T_r) = \{C' \in \mathcal C\,|\, C' \subseteq C\}$, 
	it follows that Property~(P1) must hold.
	
	To see that Property~(P2) holds, 
	Assume first that $|\mathcal C_M|=1$. Since a graph consisting of a single 
		vertex is connected, the theorem holds. So assume that $|\mathcal C_M|\geq2$.
Let $r$ and $r'$ be two roots of 
	$N$. Since $N$ is connected, there exists an undirected path in $N$ 
	between $r$ and $r'$. Let $h_1, \ldots, h_k$, $k \geq 1$, be the hybrid vertices of $N$ 
	successively crossed by that path. For all $1 \leq i \leq k$, all sets in $\mathcal C_M$ 
	corresponding to roots that are ancestors of $h_i$ form a clique in $\mathcal I(\mathcal C_M)$, 
	since they all contain the cluster $C(h_i)$. Since $r$ is an ancestor 
	of $h_1$ and $r'$ an ancestor of $h_k$, it follows that there exists a 
	path in $\mathcal I(\mathcal C_M)$ joining $C(r)$ and $C(r')$. 
	Hence, $\mathcal I(\mathcal C_M)$ is connected and Property~(P2) holds.
	
	To see that Property~(P3) holds, let $C_1$ and $C_2$ be
	two elements of $\mathcal C_M$, and let $r_1$ and $r_2$ be 
	the roots of $N$, so that $C_1=C(r_1)$ and $C_2=C(r_2)$. 
	Consider the set $H_{1,2}$ of all hybrid vertices that are 
	below both $r_1$ and $r_2$. If $H_{1,2}=\emptyset$, then $C_1 \cap C_2= \emptyset$. 
	If $|H_{1,2}| \geq 1$, then since $N$ is arboreal, there 
	exists a directed path in $N$ containing all 
	vertices in $H_{1,2}$. In particular, there is a vertex $h \in H_{1,2}$ 
	that is an ancestor of all vertices of $H_{1,2}$ in $N$. 
	This vertex $h$ satisfies $C(h)=C_1 \cap C_2$, and so $C_1 \cap C_2 \in \mathcal C$.
	Thus, Property~(P3) holds.
	
	Conversely, assume that $\mathcal C$ is a cluster system on $X$ that 
	satisfies Properties~(P1)--(P3). Then, for all $C \in \mathcal C_M$, 
	Property~(P1) implies that the set $\mathcal C_C=\{C' \in \mathcal C\,|\, C' \subseteq C\}$ 
	is a hierarchy on $C$ that contains all trivial clusters on $C$. Hence, 
	by the remark above, there 
	exists a unique (up to equivalence) phylogenetic tree $T(C)$ on $C$ 
	such that $\mathcal C(T(C))=\mathcal C_C$.
	Put $F=\{T(C)\,|\, C \in \mathcal C_M\}$, and note that $F$ need not be a phylogenetic forest on $X$ 
	since the leaf sets of the trees in $F$ might not be pairwise disjoint.
	
	We next use the trees in $F$ to recursively construct an arboreal $m$-rooted 
	network $N$ such that $\mathcal C(N)=\mathcal C$. 
	Put $F=\{T_1,\ldots, T_m\}$. First, let $N_1$ be some tree in $F$, 
	which, without loss of generality, we may assume to be
	$T_1$. Clearly, $T$ is an arboreal $1$-network. Let $1 \leq i < m$ and assume that, for all $1\leq j\leq i$, we 
	have already constructed  an arboreal $j$-network $N_j$ by processing (subject 
	to potentially having to relabel the trees in $F$) the tree $T_j\in F$. We now
	construct an arboreal $(i+1)$-network $N_{i+1}$ from $N_i$ as follows. 
	
	First, we choose a tree  $T\in F-\{T_1, \ldots, T_i\}$ such that 
	$\mathcal C (T) \cap \mathcal C(N_i) \neq \emptyset$,
	Note that it is always possible to 
	find such a tree $T$ due to the connectivity of $\mathcal I(\mathcal C_M)$ 
	that is guaranteed by Property~(P2). 
	Also note that we may assume without loss of generality that $T=T_{i+1}$. 
	Because of Property~(P3), there exists exactly one tree-vertex $u_i$ in $N_i$ 
	and one vertex $v_{i+1}$ in $T$ such that
	$C_{N_i}(u_i)=C_{T}(v_{i+1})$.
	If $u_i$ were a root of $N_i$ then since, by 
		Lemma~\ref{max-c}, $C_{N_i}(u_i)$ is a maximal cluster
		for $N_i$ it follows that  $C_T(v_{i+1})$ is also a maximal cluster of $N_i$. The definition of $F$ implies that
		$T(C_{N_i}(u_i))=T(C_T(v_{i+1}))=T$ which is impossible as $T$ has not been processed yet. So $u_i$ cannot be a root of $N_i$. 
	We then define $N_{i+1}$ as the $(i+1)$-rooted directed graph 
	obtained from $N_i$ by subdividing the incoming arc of $u_i$ in $N_i$ 
	by a vertex $w$, removing all arcs and vertices below $v_{i+1}$ in $T$, and identifying $v_{i+1}$ with $w$. 
	
	By construction, $N_{i+1}$ is clearly a $(i+1)$-network satisfying 
	$\mathcal C(N_{i+1})=\bigcup_{1 \leq j \leq i+1} \mathcal C(T_j)$. 
	Furthermore, since $N_i$ is  arboreal $N_{i+1}$ must also be arboreal.
	In particular, this implies that $N=N_m$ is a $m$-rooted network 
	satisfying $\mathcal C(N_m)=\mathcal C$. This concludes the proof. \qed\\
	
We now turn to the question of uniqueness.
Note that the construction of a network $N$ from 
a cluster system $\mathcal C$ on $X$ satisfying Properties~(P1)--(P3) as described 
in the proof of Theorem~\ref{ts-clu} requires choices to be made (e.g. the 
order in which the trees in the forest are processed in case there is a tie). As a consequence, the 
resulting network $N$ satisfying $\mathcal C(N)=\mathcal C$ need not be unique.
This issue is illustrated in Figure~\ref{inequivalent}. However, 
defining an arc in a network to be {\em bad} if both of its
end vertices are contained in $H(N)$, we  have the following result:


\begin{theorem}\label{tu-clu}
		Let $\mathcal C$ be a cluster system on $X$ that satisfies
		Properties~(P1) -- (P3) and contains all trivial clusters on $X$. 
		Then, up to equivalence, there exists a unique arboreal
		network $N$ on $X$ satisfying $\mathcal C(N)= \mathcal C$ if and only 
		if for all $C_1, C_2\in \mathcal C_M$ distinct such that  $C_1\cap C_2\not=\emptyset$
		and all $C_3\in \mathcal C_M-\{C_1,C_2\}$, we have
		$C_1 \cap C_2 \neq C_1 \cap C_3$.
		Moreover, if $N$ and $N'$ are two arboreal networks on $X$, 
		then $\mathcal C(N)=\mathcal C(N')$ if and only if $N$ and $N'$ are 
		equivalent after collapsing all bad arcs.	
\end{theorem}

This theorem is a consequence of the following lemma, its proof and Theorem~\ref{ts-clu}.
	
\begin{lemma}\label{lem:bad}
	Let $N$ be an arboreal network on $X$. Then $N$ 
	is uniquely determined by $\mathcal C(N)$ if and only if $N$ contains no bad arcs.
\end{lemma}
\begin{proof}
	Clearly, the lemma holds if $|X|=1$. So assume $|X|\geq 2$.
	Suppose first that $N$ is uniquely determined by $\mathcal C(N)$. Assume for contradiction that
	$N$ contains a bad arc $(u,v)$ with $u,v \in V(N)$. Let $p_u$ be a parent of $u$, and let $p_v$ be the 
	parent of $v$ distinct from $u$. Because $N$ is arboreal, there is no 
	directed path either from $p_u$ to $p_v$ or from $p_v$ to $p_u$. 
	Consider the network $N'$ on $X$ obtained from $N$ 
	by replacing the arcs $(p_u,u)$ and $(p_v,v)$ with the arcs $(p_u,v)$ 
	and $(p_v,u)$. Since $C_N(u)=C_N(v)$ it follows that $\mathcal C(N')=\mathcal C(N)$. 
	However, $N$ and $N'$ are not equivalent. This is a contradiction since, by assumption,
	$N$ is uniquely determined by $\mathcal C(N)$. 
	
	Conversely, suppose that $N$ does not contain a bad arc.
	Assume for  contradiction that there exists an arboreal
	network $N'$  on $X$ such that $\mathcal C(N')=\mathcal C(N)$ but $N$ and $N'$ are 
	not equivalent. Since a hybrid vertex  in a network on $X$ induces the same cluster on $X$ as its child,
	 it follows that there must exist a bijection $\chi$ between the set $T(N)$ of root and tree vertices of 
	 $N$ and the set $T(N')$ of root and tree vertices of $N'$ such that
	$C_N(v)=C_{N'}(\chi(v))$, for all $v\in T(N)$.
	Let $Comp(M)$ denote the multiply rooted 
	graph obtained from an arboreal network $M$ by collapsing all
	directed paths $P$ in $M$ that start and finish at a tree vertex of $M$ 
	and whose remaining vertices are all contained in $H(M)$. 
	Then since $N$ and $N'$ are not equivalent,
	$Comp(N)$ and $Comp(N')$ are also not equivalent. Moreover,
	$N$ must be $Comp(N)$ because $N$ does not contain a bad arc. 
	We distinguish the cases that $Comp(N')= N'$ and that $Comp(N')\neq N'$.
		
	If $Comp(N')$ is $N'$, then there must be two tree vertices $u$ and $v$ 
	in $N$ and a vertex $h\in H(N')$ such that $(u,v)$ is an arc in $N$ and
	$\chi(u),h,\chi(v)$ is a directed path in $N'$. So
	there must exist some $w\in T(N')-\{\chi(u)\}$ such 
	that $w$ is a parent of $h$. Note that $u\not=\chi^{-1}(w)$. 
	If $u$ were strictly below $\chi^{-1}(w)$ then $C_N(u)\subsetneq C_N(\chi^{-1}(w))$. 
	Hence, $C_N(v)= C_{N'}(\chi(v))=C_{N'}(\chi(u))\cap C_{N'}(w)= C_N(u)\cap C_N(\chi^{-1}(w))=C_N(u)$ 
	and so $C_N(v)=C_N(u)$, a contradiction as $(u,v)$ is an arc in $N$ and so $C_N(v)\not= C_N(u)$. 
	Similar arguments also imply that  $\chi^{-1}(w)$ cannot be 
	strictly below $u$. Since  $C_N(v)=C_{N'}(\chi(v))\subseteq C_{N'}(w)$ 
	it follows that $C_N(v)=C_N(u)\cap C_N(\chi^{-1}(w))=\emptyset$, which is again a contradiction.
		
	If $Comp(N')$ is not $N'$, then $N'$ must
	contain a bad arc, say $(h_1,h_2)$, $h_1, h_2 \in H(N)$. 
	Without loss of generality, we may assume that $h_1$ and $h_2$ 
	are such that the child $w$ of $h_2$ in $N'$ is a tree-vertex. 
	Let $r_1$ and $r_2$ be two distinct roots of $N'$ that are ancestors of $h_1$ 
	and let $r_3$ be a root of $N'$ that is an ancestor of $h_2$ but not of $h_1$. 
	Then $C_{N'}(r_1) \cap C_{N'}(r_2)=C_{N'}(r_1) \cap C_{N'}(r_3)=C_{N'}(w)$.
	It follows that there must exist some tree vertex $w_1\in V(N)$ such that 
	$C_N(\chi^{-1}(r_1)) \cap C_N(\chi^{-1}(r_2))=C_N(\chi^{-1}(r_1)) \cap C_N(\chi^{-1}(r_3))= C_N(w_1)$. 
	Consequently, the parent $p$ of $w_1$ and the parent of $p$ 
	are both hybrid vertices of $N$. But then $N$ contains a bad arc,  a contradiction.
\end{proof}

We now prove the main result of this section.

\begin{theorem}\label{lm-club}
Let $N_1$ and $N_2$ be two distinct arboreal networks on $X$ with 
$\mathcal C(N_1)=\mathcal C(N_2)$. Then $N_1$ is forest-based if and only if $N_2$ is forest-based.
\end{theorem}

\begin{proof}
If $|X|=1$, then the theorem clearly holds. So assume that $|X|\geq 2$.
Without loss of generality, it
suffices to show that if $N_1=(V_1,A_1)$ is forest-based, then $N_2=(V_2,A_2)$ must be forest-based too. 
So assume that $N_1$ is forest-based, with subdivision forest $F_1'=(V_1,A_1')$. Set $I_1=A_1-A_1'$.

By Theorem~\ref{tu-clu}, $N_1$ and $N_2$ are  equivalent after collapsing all bad arcs. 
Let $N_0$ be the graph obtained from $N_1$ in this way. Note that $N_0$ is 
not a network in our sense, as it is not semi-binary. Clearly, no arc in $I_1$ has a 
hybrid vertex as tail, so all arcs in $I_1$ are arcs of $N_0$. Since $N_0$ 
can also be obtained by collapsing all bad arcs of $N_2$, this induces a 
trivial bijection $\chi$ between $I_1$ and some set $I_2$ of arcs of $N_2$.

It remains to show that the forest $F_2'=(V_2,A_2-I_2)$ is a subdivision 
forest for $N_2$. Clearly, we have that $L(N_2) \subseteq L(F_2')$, and 
since $N_2$ is arboreal, no arc of $I_2$ joins two vertices from the same tree 
in $F_2'$. To see that $L(F_2') \subseteq L(N_2)$ holds too, assume for contradiction that 
there is a vertex $v_2\in L(F_2')-L(N_2)$. Then  all arcs of $N_2$ with tail $v_2$ are in $I_2$. 
Note that $I_2$ has been defined in such a way that no arc in $I_2$ is collapsed 
when transforming $N_2$ into $N_0$. Moreover, the property of having a vertex 
such that all outgoing arcs belong to a given set is preserved when resolving vertices of 
a network. Since $\chi$ is the trivial bijection between $I_1$ and $I_2$, it follows that, there 
must exist a vertex $v_1$ in $N_1$ such that all arcs of $N_1$ 
with tail $v_1$ are in $I_1$. This is a contradiction since $F_1'=(V_1,A_1')$ is a subdivision 
forest for $N_1$ and so $L(F_1')=L(N_1)$. Hence, $N_2$ is forest-based.
\end{proof}

\section{Characterizing proper forest-based networks}
\label{sec:charaterize}

In this section we present two characterizations for proper forest-based  
networks (Theorems~\ref{theo:general} and \ref{col2}). 
Various characterizations have been given for tree-based 
{\em phylogenetic} networks (see e.g. \cite{FS15} and \cite[Theorem 10.17]{S16}).
Some of these are given in terms of bipartite graphs, one of
which from \cite{JvI18} we now recall. Define a vertex in a network 
$N$ to be an \emph{omnian (vertex)} of $N$ if all 
of the children of $v$ are contained in $H(N)$, and let 
$\cO(N)$ denote the set of omnians in $N$ (see 
e.\,g.\, Figures~\ref{fig:nonproper}(i) and \ref{thetaN}(i)). To a 
network $N$ associate the bipartite graph $(U \cup H, E)$,
where $U$ contains a vertex $u_v$ for each omnian $ v \in \cO(N)$, 
$H$ contains a vertex $u_w$ for each hybrid vertex  $w \in H(N)$,
and $E$ consists of the edges $\{u_v,u_w\}$ such that there is a
some $v \in \cO(N)$ and some  $w \in H(N)$ with $(v,w)$ an arc in $N$.
Then, a {\em phylogenetic} network is tree-based if and only if $(U \cup H, E)$
contains a matching \cite[Theorem 3.4]{JvI18}.
Interestingly, we found that characterizing forest-based networks
is more subtle although, as we shall now see, we can 
still characterize proper forest-based networks using omnians.

To this end, we introduce some further definitions. 
Suppose that $N$ is a network 
and that $v \in V(N)$. We define the vertex
$\gamma_v\in RH(N)=R(N) \cup H(N)$ to be
the (unique) ancestor of $v$ such that 
no vertex in $RH(N)-\{\gamma_v\}$
is contained in the directed path $P$ 
from $\gamma_v$ to $v$ (e.g. in 
Figure~\ref{fig:nonproper}(i) 
\begin{figure}[h]
	\begin{center}
		\includegraphics[scale=0.6]{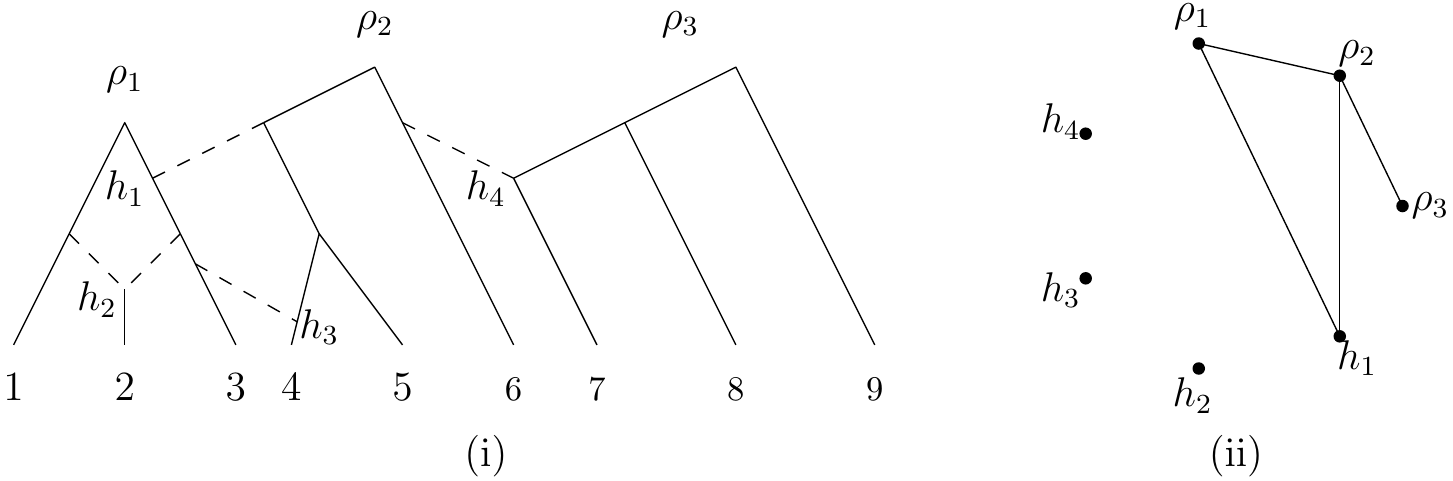}
		\caption{ (i) A 3-rooted forest-based network $N$ with $\cO(N)=\emptyset$ that is not proper forest-based. (ii)
				The graph $\Gamma(N)$.
		}
		\label{fig:nonproper}
	\end{center}
\end{figure}
for leaf 5, $\gamma_5=\rho_2$, and for leaf 3, $\gamma_3=h_1$).
Note that $\gamma_v=v$ if and only if $v \in RH(N)$. 
The rational behind the definition of $\gamma_v$  is that, 
for any base forest $F$ in a proper forest-based 
network, the vertices $v$ and $\gamma_v$ 
must belong to the same tree in $F$.

We next associate an undirected graph $\Gamma(N)$ to $N$
(which may also contain loops).
The vertex set of $\Gamma(N)$ is the set $RH(N)$, and 
(not necessarily distinct)  vertices $u, v  \in RH(N)$ form an edge 
$\{u,v\}$ in $\Gamma(N)$ if there exists a hybrid vertex $h\in H(N)$ with
parents $u'$ and $v'$ such that $u=\gamma_{u'}$ and
$v=\gamma_{v'}$ (see e.g. Figure~\ref{fig:nonproper}(ii)).
In addition, we call any (undirected) supergraph $\Gamma'(N)$ of
$\Gamma(N)$ with the same vertex set as $\Gamma(N)$
an \emph{omni-extension} of $\Gamma(N)$ if,
for any omnian $v \in \cO(N)$, there exists a
child $h$ of $v$ such that $\{\gamma_u, h\}$
is an edge of $\Gamma'(N)$ for $u$ the second parent of $h$ (see 
e.g. Figure~\ref{thetaN}(iii)).
\begin{figure}[h]
	\begin{center}
		\includegraphics[scale=0.6]{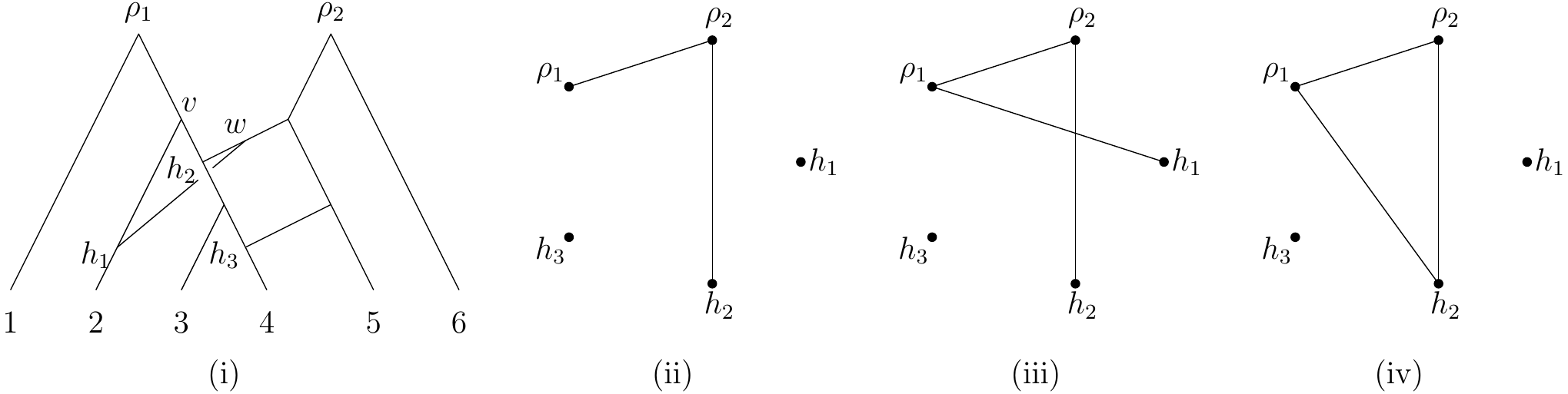}
		\caption{
			(i) A 2-rooted forest-based network $N$ with $\cO(N)=\{v,w\}$. 
			(ii) The graph $\Gamma(N)$. 
			(iii) and (iv) Two distinct omni-extensions of $\Gamma(N)$,  
			with a minimum number of possible edges. Since one of these 
			extensions has no cycle of length 3, it 
			follows by Theorem~\ref{col2} that $N$ is proper forest-based.}
		\label{thetaN}
	\end{center}
\end{figure}
Note that there exist networks $N$ on $X$ such that $\Gamma(N)$ has more than
one omni-extension (e.g. Figure~\ref{thetaN}), 
and also that if  $N$ does not contain any 
omnians, then  $\Gamma(N)$ is an omni-extension of itself (this can 
also hold even if $N$ contains omnians). 
We will use the following useful additional 
observation concerning omni-extensions to obtain our characterisation 
of proper forest-based networks.

\begin{lemma} \label{lem:new}
	Let $N$ be a $m$-rooted network on $X$, some $m\geq 2$.
	If $N$ is proper forest-based with proper base forest $F$, then $\Gamma(N)$
	has an omni-extension that does not contain loops, namely,
	the graph $\Gamma_F(N)$ having the same vertex set as $\Gamma(N)$, and with
	edge set consisting of those $\{u,v\}$, $u,v\in RH(N)$, such that  
	$u$ and $v$ belong to different trees in $F$.
\end{lemma}
\begin{proof}
	We first establish that $\Gamma(N)$ 
	is a subgraph of $\Gamma_F(N)$. 
	Suppose $e=\{u,v\}$ with $u,v\in RH(N)$ is an edge in $\Gamma(N)$.
	 Then there is a hybrid vertex
		in $N$ with parents $u'$ and $v'$ such that
		$u=\gamma_{u'}$ and $\gamma_{v'}=v$. Since 
		$N$ is based on $F$, 
	$u'$ and $v'$ must belong to two different trees in $F$. 
	Since $F$ is a proper base forest for $N$, the vertices $\gamma_{u'}$ and $\gamma_{v'}$ must
	belong to two different trees in $F$. Thus, $u\not= v$. By 
		definition of $\Gamma_F(N)$, it follows that $e$ 
		is an edge of $\Gamma_F(N)$.
		
	To show that $\Gamma_F(N)$ 
	is an omni-extension of $\Gamma(N)$, consider an omnian $v$ of $N$.
	As $N$ is based on $F$, $v$ must have at least 
	one child $h$ such   
	that $v$ and $h$ belong to the same tree $T_v$ of $F$. 
	Hence, for $u$ the parent of $h$ other
		than $v$, $u$ does not belong to $T_v$. Thus, 
	$\gamma_u$ and $h$ belong to different trees in $F$
	because $F$ is a proper base forest for $N$. Hence
		$\{\gamma_u,h\}$ is an edge of
	$\Gamma_F(N)$. 
		\end{proof}
	
We now present our characterization for proper forest-based networks.	
Recall that if $G$ is a undirected graph (possibly with loops), and 
$Y$ is a non-empty set of colors, then a map $\sigma:  V(G) \to  Y$ satisfying $\sigma(u) \neq \sigma(v)$ 
for all edges $\{u,v\}$ of $G$ is a \emph{proper vertex coloring} of $G$. 
Moreover, if there exists such a colouring with $|Y|=k \ge 1$, then $G$ is called \emph{$k$-colorable}; if
$k=2$ then $G$ is {\em bipartite}.
	
\begin{theorem}\label{theo:general}
	Let $N$ be a $m$-rooted network on $X$, some $m\geq 2$, and let
	$\{s_1,\ldots, s_m\}$ be a set of $m$ colors. Then $N$ is proper forest 
	based if and only if there exists an omni-extension $\Gamma'(N)$ of $\Gamma(N)$ 
	and a proper vertex coloring $\sigma: RH(N) \to \{s_1, \ldots, s_m\}$ of $\Gamma'(N)$ satisfying:
	\begin{itemize}
		\item[(C1)] The restriction of $\sigma$ to $R(N)$ is a bijection.
		\item[(C2)] For all $u \in R(N)$ and all $ v \in H(N)$
		such that $\sigma(u)=\sigma(v)$ there must
		exist a directed path $P$ in $N$ from $u$ to $v$ such that $\sigma(w)=\sigma(u)$ 
		holds for all vertices $w \in H(N)$ that lie on $P$.
	\end{itemize}
\end{theorem}

\begin{proof}
	Assume first that $N$ is proper
	forest-based with proper base forest $F$. Let $\Gamma_F(N)$ 
	be the omni-extension of $\Gamma(N)$ given in  Lemma~\ref{lem:new}.
	Let $\sigma_F: RH(N) \to R(N)$ be 
	the map that assigns to every vertex $v \in RH(N)$ the unique root $\rho$ of $N$ such that 
	the directed path from $\rho$ to $v$ does not contain a contact arc of $N$. 
	Note that such a path may consist of a single vertex. 
	Since $N$ is based on $F$, it follows that $\sigma_F$ is 
	well-defined and a proper vertex coloring of $\Gamma_F(N)$. By definition,
	$\sigma_F$ satisfies Properties~(C1) and (C2).
	
	Conversely, let $\Gamma'(N)$ be an omni-extension of $\Gamma(N)$,
	let $S=\{s_1, \ldots, s_m\}$ denote a set of $m$ colors, 
	and let $\sigma: RH(N) \to S$ be a proper vertex 
	coloring of $\Gamma'(N)$ that satisfies 
	Properties~(C1) and (C2). For $1 \leq i \leq m$, let $T_i$ denote the 
	subgraph of $N$ induced on the set $V'$ of vertices $v$ in $N$ with $\sigma(\gamma_v)=s_i$
	(that is, the graph with vertex set $V'$ and arc set $\{(u,v) \in A(G): u, v \in V'\}$).
	
	Suppose $i\in\{1,\ldots, m\}$.
	We claim that $T_i$ is a subdivision of a phylogenetic tree $T'_i$ on 
	some subset of $X$. By symmetry, we may assume without 
	loss of generality that $i=1$. Since $N$ has $m$ roots and 
	since, by Property~(C1), no two roots of $N$ are assigned the 
	same color under $\sigma$, it follows that $T_1$ contains 
	exactly one root of $N$. Moreover, and as a direct consequence of Property~(C2), we have that $T_1$ is connected.
		
	To see that $T_1$ is a tree, it suffices to show that $T_1$ does not 
	contain a hybrid vertex of $N$ and both its parents. Assume for contradiction that $T_1$ contains 
	a hybrid vertex $h\in H(N)$ and its parents $u$ and $v$. By 
	definition of $\Gamma(N)$, $\{\gamma_u, \gamma_v\}$ is an edge of $\Gamma(N)$.
	Since $\sigma$ is a proper vertex coloring of $\Gamma'(N)$
	it follows that $\sigma(\gamma_u) \neq \sigma(\gamma_v)$. 
	This is a contradiction since $u,v\in V'$ and, therefore,
	$\sigma(\gamma_u)= s_1=\sigma(\gamma_v)$. 
    Thus, $T_1$ must be a tree, as required. 
	
	Since, $\Gamma'(N)$ is an omni-extension 
	of $\Gamma(N)$, the definition of $\sigma$ ensures
	that  $L(T_1) \subseteq X$. It follows that $T_1$ is a subdivision of 
	a phylogenetic tree $T'_1$ on a subset of $X$, as claimed.
	
	Now let $F=\{T_1, \ldots, T_m\}$. Then,
	by construction,  we have  
	$L(T_i)\not=L(T_j)$, for all $1\leq i<j\leq m$.
	In view of our claim, every tree in $F$ is a subdivision of a 
	phylogenetic tree in the forest $F'=\{T'_1, \ldots, T'_m\}$ and 
	$\bigcup_{T\in F'} L(T)=X$. Moreover, for all $i \in \{1, \ldots, m\}$, 
	an arc $(u,v)$ of $N$ with $u,v \in V(T_i)$ is also an arc of $T_i$. It 
	follows that $N$ is obtained from $F$ by adding arcs joining 
	vertices from distinct trees of $F$.
	Thus, $N$ is forest-based. That $N$ is proper
	forest-based is a direct consequence of the construction of $F$ from $N$.
\end{proof}

Interestingly, Theorem~\ref{theo:general} can be strengthened in case $m=2$ as follows. 

\begin{theorem}\label{col2}
	Let $N$ be a 2-rooted network on $X$. Then $N$ is proper forest-based
	if and only if $\Gamma(N)$ has a bipartite omni-extension.
\end{theorem}
\begin{proof}
Suppose that $N$ is proper forest-based $2$-network 
with proper base forest $F=\{T_1,T_2\}$. Then, by Theorem~\ref{theo:general}, there 
exists an omni-extension $\Gamma'(N)$ that is 2-colorable. 
	 				
Conversely, suppose that there exists 
an omni-extension $\Gamma'(N)$ of $\Gamma(N)$ that is 2-colorable.
Then there exists a proper vertex colouring $\sigma:RH(N)\to \{s_1,s_2\}$, with $s_1 \neq s_2$.  
In view of Theorem~\ref{theo:general}, it suffices to show 
that $\sigma$ satisfies Properties~(C1) and (C2).  

Since $N$ is connected, there must exist some hybrid vertex 
$h\in H(N)$ with parents $u',v'$ satisfying $\gamma_{u'}=\rho_1$ and $\gamma_{v'}=\rho_2$.
So $\{\rho_1,\rho_2\}$ is an edge in $\Gamma(N)$, and  
therefore $\sigma(\rho_1)\not=\sigma(\rho_2)$ since $\sigma$ is a proper vertex 
colouring of $\Gamma'(N)$. Thus Property~(C1) holds.

To see that Property~(C2) holds,  consider the
map $\psi=\psi_{\sigma}: V(N) \to \{\rho_1, \rho_2\}$ 
associated to $\sigma$ given by putting, for all
$v\in V(N)$, 
$\psi(v)=\sigma(\gamma_v)$.
Assume for contradiction that (C2) does not hold.
Then there must exist some $i\in\{1,2\}$, say
$i=1$, and some vertex $g\in H(N)$ with $\sigma(\rho_1)=\sigma(g)$
such that every directed path from $\rho_1$ to $g$  in $N$ contains a vertex $r'\in H(N)$ 
for which $\sigma(r')\not=\sigma(\rho_1)$. 
Let $P$ denote a directed path from $\rho_1$ to $g$.
Without loss of
generality, we may assume that $r'\in H(N)$ is a vertex on $P$ 
such that, for every
vertex $w\in V(N)$ on $P$ strictly above $r'$, we have
$\psi(w)=\sigma(\rho_1)$. Furthermore, we may assume without
loss of generality that $g$ is such that, for every 
$z\in V(N)$ on $P$ that is strictly above $g$ but below $r'$, we have
$\psi(z)=\sigma(\rho_2)$. 

Let $r\in V(N)$ denote the parent 
of $g$ on $P$ and let $g'\in V(N)$ 
denote the parent of $r'$ on $P$.
Let $q\in V(N)$ denote the other parent of $g$. 
Then, by definition of $g$, it follows that
$\{\gamma_{r},\gamma_q\}$ must be an edge in $\Gamma'(N)$. 
Since, by assumption, $\Gamma'(N)$ does not contain a cycle of  
length one (as otherwise $\Gamma'(N)$ would not be 2-colorable), 
it follows that  $\gamma_{r}\not=\gamma_q$.
Hence, $\sigma(\rho_2)=\psi(r)=\sigma(\gamma_{r})\not = \sigma(\gamma_q)$, and
so $\sigma(\gamma_q)=\sigma(\rho_1)$
because $\sigma$ is a 2-colouring. If 
$\gamma_q$ is a vertex on $P$ above $g'$, 
we obtain a contradiction, since the definition of $g$ implies that
we have found a directed path $P'$
from $\rho_1$ to $g$ in $N$ 
such that $\sigma(w)=\sigma(\rho_1)$ for all
vertices $w\in H(N)$ contained in $P'$.
By the choice of $g$, it follows that, $\gamma_q$ does not lie on $P$. 
Similar arguments as in the case of $g$, 
$r$, and $q$ imply that for one of the parents of $\gamma_q$
in $N$, $z$ say, we also have  $\sigma(\gamma_z)=\sigma(\rho_1)$.
Repeating this argument,  
since $V(N)$ is finite, we eventually obtain a  directed path $P^*$ from $\rho_1$ to $g$ in $N$ such that
$\sigma(\rho_1)=\sigma(h)$ holds for  every hybrid vertex $h$ on $P^*$,  a contradiction. Thus, Property~(C2) must hold.
\end{proof}

Note that the 
network in Figure~\ref{fig:nonbinary}(i) shows that the assumption that
the network in Theorem~\ref{col2} is semi-binary is necessary
(since, extending relevant definitions for semi-binary networks in the obvious way 
to general networks $N$ in which not every hybrid vertex must have indegree two,
every omni-extension of $\Gamma(N)$  is a supergraph of $\Gamma(N)$, and 
$\Gamma(N)$ contains a cycle of length three).
Also, the network $N$ depicted in Figure~\ref{fig:nonproper}(i) shows
that Theorem~\ref{col2} need not hold for $m$-rooted networks with $m \geq 3$,
since $\Gamma(N)$ is an omni-extension of itself because $\frak{O}(N)=\emptyset$, and $\Gamma(N)$ is not bipartite.

\begin{figure}[h]
	\begin{center}
		\includegraphics[scale=0.7]{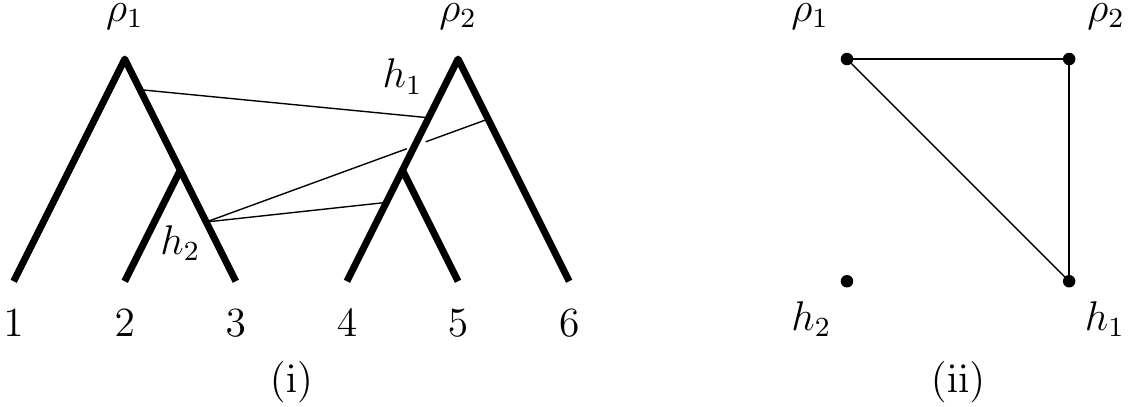}
		\caption{(i) A directed, acyclic graph $N$ with two roots
			that is based on the forest indicated in
			bold edges, and (ii) the graph $\Gamma(N)$. 
			Note that $N$ is not semi-binary as $indeg_N(h_2)=3$.}
		\label{fig:nonbinary}
	\end{center}
\end{figure}

\section{Universal forest-based networks}
\label{sec:universal}

It has been shown in \cite{H16} and \cite{Z16} that 
there exist tree-based, binary phylogenetic networks $N$ on $X$ such 
every possible binary phylogenetic tree on $X$ is a base-tree for $N$.
Such binary networks are called \emph{universal tree-based} networks.
It is thus of interest to understand if there are binary universal forest-based networks
(i.e. binary networks $N$ such that every phylogenetic forest on $X$ is a base forest for $N$).
In case $|X| \le 3$ there always exists such a network
(see Figure~\ref{ufb3} for $|X|=3$).

\begin{figure}[h]
	\begin{center}
		\includegraphics[scale=0.6]{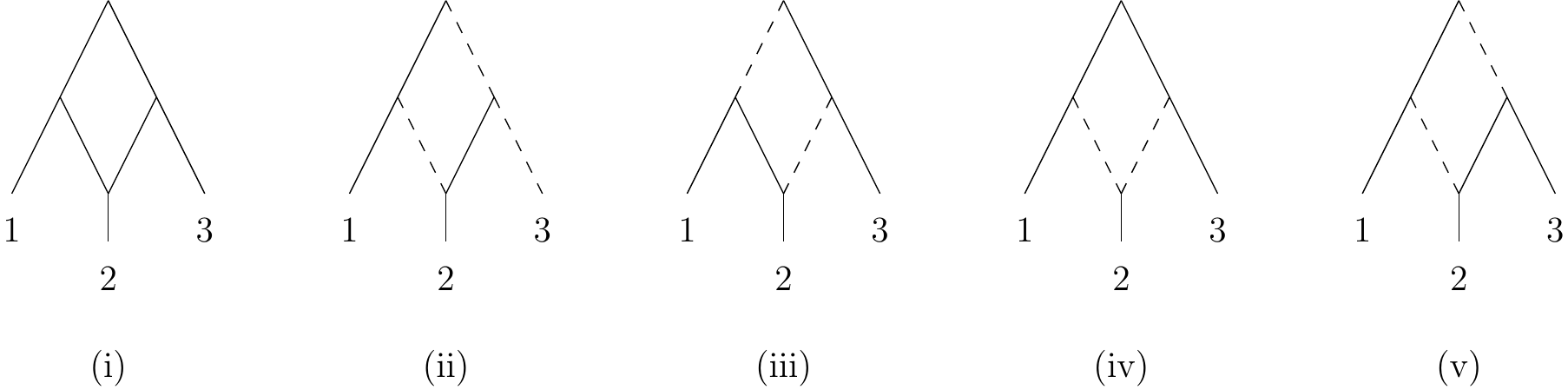}
		\caption{(i) A universal forest-based network on 
			 $X=\{1,2,3\}$. 
			 (ii) -- (v) 
			Embeddings of the four phylogenetic forests on $X$ into the network in (i). In all cases, 
			the dashed arcs represent contact arcs.}
		\label{ufb3}
	\end{center}
\end{figure}

However, we now prove the following:

\begin{theorem}\label{no-universal}
	For all $X$ with $|X|\ge 4$, there does not exist a universal forest-based network on $X$.
\end{theorem}

To prove this theorem we begin with a useful observation.

\begin{lemma}\label{helpful}
	Suppose that $U$ is a universal forest-based network on $X$, $|X| \ge 4$.
	Then, for $x,y \in X$ distinct, and all $p,q \ge 0$, $U$ does not contain the configuration pictured in 
	Figure~\ref{badconfig}, where $v_0=x$ and $q_0=y$.
\end{lemma}
\begin{proof}
Since $U$ is universal forest-based and $|X|\geq 4$ there must exists a base forest $F$ for $U$ that 
has a component $T$ which has two leaves $x,y \in X$ so that $x$ and $y$ are
not contained in two arcs in $T$ that have a common tail. In particular, $T$ has at least 3 leaves. 
Let $F'$ be some embedding of $F$ in $U$ and let $T'$ be the corresponding 
embedding of $T$, which exists as $U$ is universal.

Put $x=v_0$ and $y=q_0$. Let $w$ denote a tree vertex or a root of $U$ and, 
for all $1\leq i\leq p$ and all $1\leq j\leq q$, let $v_i$, and $u_j$ denote hybrid 
vertices of $U$ such that $v_i$ is the parent of $v_{i-1}$ and $u_j$ is the
parent of $u_{j-1}$ and $w$ is the parent of $v_p$ and of $u_q$ (see the
configuration depicted in Figure~\ref{badconfig}).  
Then since all $v_i$, $1\leq i\leq p$, and all $u_j$, $1\leq j\leq q$, are hybrid vertices of $U$ it follows
that $x=v_0,\dots,v_p, y=u_0,\dots, u_q$ 
must all be contained in $T'$. But then at least one
of $(w,v_p)$ and $(w,u_q)$ must be an arc in $T'$, otherwise $w$
would be a leaf of some component in $F'$ and so $L(F)\not=L(U)$. 
Since $w$ is a tree vertex or a root of $U$ this
implies that $w,v_p,u_q$ are vertices in $T'$. Thus, both arcs
$(w,v_p)$ and $(w,u_q)$ must be arcs in $T'$. But this implies 
that $x$ and $y$ are	contained 
in two arcs in $T$ that have a common tail, a contradiction.
\end{proof}

\begin{figure}[h]
	\begin{center}
	 \includegraphics[scale=0.7]{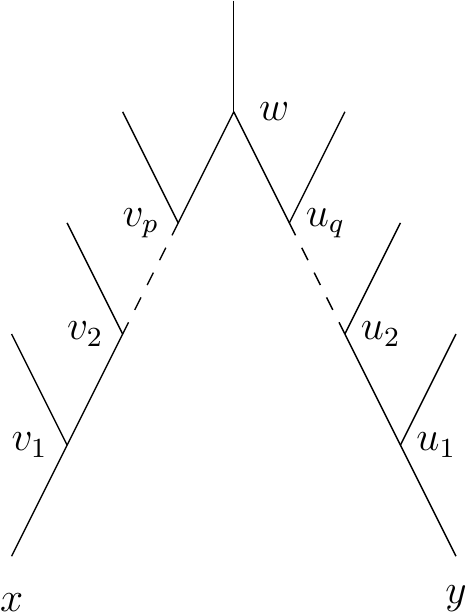}
		\caption{Forbidden configuration in a universal network. Note that $w$ can also be a root vertex.}
		\label{badconfig}
	\end{center}
\end{figure}

\noindent{\em Proof of Theorem~\ref{no-universal}:}
Assume for contradiction that there exists a universal forest-based network $U$ on $X$. 
Let $w$ be a root or tree vertex of $U$ such that all non-leaf vertices 
below $w$ are contained in $H(U)$. 
Note that this configuration must exist since $|X|\geq 4$, and so 
there are at least two base forests on $X$. Let $u$ and $v$ be the
children of $w$. By Lemma~\ref{helpful}, there exists a unique leaf $x \in X$ of $U$ such that $x$ is a
descendant of $w$. In particular, $x$ is a descendant of both $u$ and $v$. Now let $F$ be
a forest with  two components, one of which is the tree $T_x$ whose sole 
vertex is $x$ and the other which is the phylogenetic tree $T$ where
$T$ has leaf-set $X-\{x\}$. Let $F'$ be an embedding of $F$ in $U$ and
$T'_x$ be the corresponding embedding of $T_x$ into $U$ (which exists
as $U$ is universal). Note that $T'_x$ is a directed path ending at $x$.

Since the directed paths from $u$ to $x$ and from $v$ to $x$ only contain hybrid vertices, 
$T_x'$ must contain both of these paths. But this is a contradiction, since the union of
these paths must contain a hybrid vertex and both its parents. \qed

\section{Conclusion}
\label{sec:conclusion}

In this paper we have introduced the concept of 
forest-based networks and investigated some of
their fundamental properties. We conclude by 
indicating some possible future directions of research for
forest-based networks.

In Section~\ref{sec:relationship} we studied the relationship between
forest-based networks and other classes of networks. It
could be interesting to investigate these relationships in more detail.
For example,  it is know that binary tree-child phylogenetic networks 
are precisely the tree-based networks such that
every embedded phylogenetic tree is a base tree \cite{semple2016phylogenetic} -- does
a similar result hold for forest-based networks?
In addition, in this paper we only considered properties of 
semi-binary networks. Which of our definitions and 
results extend to non-binary networks (i.e. networks that are not 
necessarily semi-binary)?  Note that in \cite{JvI18} properties of non-binary 
tree-based networks were considered, which might 
provide some useful leads to studying this question.

There are also several open algorithmic questions that could
be investigated. For example, there are efficient algorithms for deciding 
whether a given phylogenetic network is tree-based 
or not, and if so to find a base-tree \cite{FS15, JvI18}.
Is there an efficient algorithms for deciding 
whether a given phylogenetic network is forest-based or not?
In this regards, Theorem~\ref{theo:general} might
be useful as it could provide a useful link with 
coloring problems.  It is also known to be NP-complete to decide 
whether or not a binary phylogenetic network is based
on a given binary phylogenetic tree -- does a similar result
hold for forest-based networks?

Finally, it could be interesting to study 
related classes of networks. For example,
{\em pedigrees} \cite{steel2006reconstructing}
are closely related to multiply rooted networks,
and it is known that the two subgraphs of a pedigree induced 
by the bipartition of the pedigree into its male and female individuals 
are both forests \cite[Lemma 1.4.4]{SS03}.
Are there interesting relationships 
between pedigrees and forest-based networks? 
Also, we could consider a generalisation of tree-based
{\em unrooted} phylogenetic networks 
which were first considered in \cite{francis2018tree}.
In particular, an {\em unrooted forest-based network} $N$
is an unrooted phylogenetic network on $X$ (as defined in \cite{francis2018tree})
that contains a spanning forest with leaf-set $X$ such that 
no edge in $N$ has both of its vertices in the same tree of the forest.
What properties do such networks  enjoy? 


\bibliographystyle{abbrvnat}
\bibliography{bibliography}

\end{document}